\documentclass[12pt, reqno]{amsart}
\usepackage{amssymb, amstext, amscd, amsmath}
\usepackage{color}

\usepackage{xy}
\xyoption{all}

\makeatletter
\def\@cite#1#2{{\m@th\upshape\bfseries%
[{#1\if@tempswa{\m@th\upshape\mdseries, #2}\fi}]}} \makeatother
\theoremstyle{plain}
\newtheorem{thm}[subsection]{Theorem}
\newtheorem{cor}[subsection]{Corollary}
\newtheorem{prop}[subsection]{Proposition}
\newtheorem{lem}[subsection]{Lemma}
\theoremstyle{definition}
\newtheorem{rem}[subsection]{Remark}

\newtheorem{defn}[subsection]{Definition}

\newtheorem{eg}[subsection]{Example}

\newcommand{\bA}{{\mathbb{A}}}
\newcommand{\bB}{{\mathbb{B}}}
\newcommand{\bC}{{\mathbb{C}}}

\newcommand{\bF}{{\mathbb{F}}}

\newcommand{\bN}{{\mathbb{N}}}

\newcommand{\bT}{{\mathbb{T}}}

  \newcommand{\A}{{\mathcal{A}}}
  \newcommand{\B}{{\mathcal{B}}}
  \newcommand{\C}{{\mathcal{C}}}
  
  \newcommand{\E}{{\mathcal{E}}}
  \newcommand{\F}{{\mathcal{F}}}
  
\renewcommand{\H}{{\mathcal{H}}}

  \newcommand{\K}{{\mathcal{K}}}
\renewcommand{\L}{{\mathcal{L}}}

\renewcommand{\O}{{\mathcal{O}}}

  \newcommand{\R}{{\mathcal{R}}}

  \newcommand{\W}{{\mathcal{W}}}

\newcommand{\fA}{{\mathfrak{A}}}

\newcommand{\fS}{{\mathfrak{S}}}

\newcommand{\upchi}{{\raise.35ex\hbox{\ensuremath{\chi}}}}

\newcommand{\id}{{\operatorname{id}}}

\newcommand{\spn}{\operatorname{span}}

\newcommand{\Fn}{\mathbb{F}_n^+}

\newcommand{\Fock}{\ell^2(\Fn)}

\newcommand{\mt}{\varnothing}

\newcommand{\ol}{\overline}

\newcommand{\sot}{\textsc{sot}}
\newcommand{\wot}{\textsc{wot}}

\newcommand{\vn}{\text{vN}}

\newcommand{\otl}{\, \ol\otimes\, }

\begin{document}
\title[Free Semigroup Algebras and Hopf Algebras]
{Analytic Free Semigroup Algebras \\ and Hopf Algebras}
\author[D. Yang]
{Dilian Yang}
\address{Dilian Yang,
Department of Mathematics $\&$ Statistics, University of Windsor, Windsor, ON
N9B 3P4, CANADA} \email{dyang@uwindsor.ca}

\begin{abstract}
Let $\fS$ be an analytic free semigroup algebra. In this paper, we explore richer structures of
$\fS$ and its predual $\fS_*$.
We prove that $\fS$ and $\fS_*$ both are Hopf algebras.
Moreover, the structures of $\fS$ and $\fS_*$ are closely connected with each other:
There is a bijection between the set of completely bounded representations of $\fS_*$ and
the set of corepresentations of $\fS$ on one hand,
and $\fS$ can be recovered from the coefficient operators of completely bounded representations of $\fS_*$
on the other hand. As an amusing application of our results, the (Gelfand) spectrum of $\fS_*$ is identified.
Surprisingly, the main results of this paper seem new even in the classical case.
\end{abstract}

\subjclass[2000]{47L80, 16Txx.}
\keywords{free semigroup algebra, Hopf algebra, completely bounded representation, corepresentation.}
\thanks{The author was partially supported by an NSERC Discovery grant.}

\date{}
\maketitle


\section{Introduction}

Let $S_i$ ($i=1,...,n$) be operators acting on a Hilbert space $\H$. The $n$-tuple $S=[S_1\,\cdots\,S_n]$ is said to be isometric
if it is an isometry from $\H^n$ to $\H$, equivalently,
\[
S_i^*S_j=\delta_{i,j}I\quad (i,j=1,...,n).
\]
On one hand, it is well known that the C*-algebra generated
by an isometric $n$-tuple $S$ is isomorphic to the Cuntz algebra $\O_n$ if $\sum_iS_iS_i^*=I$,
and isomorphic to the Cuntz-Toeplitz algebra $\E_n$ if $\sum_iS_iS_i^*<I$.
On the other hand, the unital norm closed (non-selfadjoint) operator algebra generated by $S$
is completely isometrically isomorphic to Popescu's non-commutative disk algebra $\A_n$  (cf. \cite{Pop96}).
Because of those rigidities,
Davidson and  Pitts initiated the study of free semigroup algebras in \cite{DavPit99}
in order to study the fine structure of isometric $n$-tuples.
Since then, free semigroup algebras have attracted a great deal of attention.
See, for example, \cite{Dav2, Dav1, DavKatPit, DKS, DLP, DavPit, DavPit98, DavPit99,
 DavWri,  DavYang, Ken0, Ken}.
It should be mentioned that the structure of isometric tuples has been completely described
by Kennedy in \cite{Ken} very recently. It turns out that an isometric tuple has a higher-dimensional Lebesgue-von Neumann-Wold decomposition.

A free semigroup algebra $\fS$ is the unital $\wot$-closed (non-selfadjoint) operator algebra generated by
an isometric tuple.
The prototypical example of free semigroup algebras is the non-commutative analytic Toeplitz algebra $\L_n$ generated
by the left regular representation of the free semigroup $\Fn$ with $n$ generators.
The algebra ${\L_n}$ was introduced in \cite{Pop91} by Popescu  in connection with a non-commutative von Neumann inequality, and
plays a prominent role in the theory of free semigroup algebras.
It is singled out to name a class of free semigroup algebras analytic:
A free semigroup algebra $\fS$ generated by $S$ is said to be analytic if it is algebraically isomorphic to $\L_n$.
Analyticity plays a fundamental role in the existence of wandering vectors. Recall that a wandering vector for
an isometric tuple $S$ acting on $\H$ is a unit vector $\xi$ in $\H$ such that the set $\{S_w\xi:w\in\Fn\}$ is orthonormal.
Here $S_w=S_{i_1}\cdots S_{i_k}$ for a word $w=i_1\cdots i_k$ in $\Fn$.
It was conjectured in \cite{DavKatPit} that every
analytic free semigroup algebra has wandering vectors. This conjecture has been settled in \cite{Ken0} recently.
It is shown there that a free semigroup algebra either has a wandering vector, or is
a von Neumann algebra.

The non-commutative disk algebra $\A_n$  is the unital norm closed algebra generated by
the left regular representation of $\Fn$.
It was also introduced by Popescu in \cite{Pop91}.
He showed in \cite{Pop96} that the unital norm closed algebra $\fA_n$ generated by an arbitrary isometric $n$-tuple
is completely isometrically isomorphic to $\A_n$.
The algebra $\A_n$ also plays an important role in studying free semigroup algebras.
The free semigroup algebra $\fS$ generated by an isometric $n$-tuple $S$ is said to be absolutely continuous
if the representation of $\A_n$ induced by $S$ can be extended to a weak*-continuous representation of
$\L_n$. This notion was introduced in \cite{DLP} in order to obtain a natural analogue of
the measure-theoretic notion of absolute continuity, and played a central role there.

From the definitions and \cite[Theorem 1.1]{DavKatPit} (cf. Theorem \ref{T:DKP} below), one sees that analyticity
implies absolute continuity. The converse fails in the classical case
(i.e., when $n=1$). Refer to, for example, \cite{Ken, Wer}. However it holds true in higher-dimensional
cases \cite{Ken}.

Let $\fS$ be an analytic free semigroup algebra generated by an isometric $n$-tuple.
The main purpose of this paper is to explore richer structures of
$\fS$ and its predual $\fS_*$ (recall from \cite{DavWri} that $\fS$ has a unique predual), so that
one could study those objects from other points of view. This could lead us to understand them better and deeper.
More precisely, we show that both $\fS$ and $\fS_*$ are Hopf algebras (Theorems \ref{T:abscon} and  \ref{T:Conv}).
Moreover, the structures of $\fS$ and $\fS_*$ are closely connected with each other.
We prove that there is a bijection between
the set of completely bounded representations of $\fS_*$ and
the set of corepresentations of $\fS$ (Theorem \ref{T:corep}),
and that the following duality result holds true:
$\fS$ can be recovered from the coefficient operators of completely bounded representations of $\fS_*$ (Theorem \ref{T:dual}).
Furthermore, as an amusing application of Theorem \ref{T:corep},
we show that  the (Gelfand) spectrum of $\fS_*$ is precisely $\Fn$ (Theorem \ref{T:char}).
Rather surprisingly, according to the author's best knowledge,
the main results of this paper seem new even in the classical case,
that is, when $\fS$ is the analytic Toeplitz algebra.

It should be pointed out that this work connects with harmonic analysis.
It is shown in \cite{DavPit99, DavWri} that $\L_n$ has a unique predual
\[
{\L_n}_*=\{[\xi\eta^*]: \xi,\eta\in\Fock\},
\]
where $[\xi\eta^*]$ is the rank one linear functional given by $[\xi\eta^*](A)=(A\xi, \eta)$ for all $A\in\L_n$.
If we regard $\L_n$ as a non-selfadjoint analogue of the group von Neumann algebra $\vn(\bF_n)$ of the free group $\bF_n$
with $n$ generators, then ${\L_n}_*$ is analogous to  the Fourier algebra $A(\bF_n)$,
the predual of $\vn(\bF_n)$. See, for example, \cite{Eym} for more information about Fourier algebras.
So one could think of ${\L_n}_*$ as a sort
of ``Fourier algebra of the free semigroup $\Fn$''.
To seek an algebra structure for ${\L_n}_*$ is actually the starting motivation of this work.
This point of view, in my opinion, could be fruitful in further studying free semigroup algebras.
In this paper, we shall  see certain analogy between ${\L_n}_*$
and $A(\bF_n)$.

The rest of the paper is organized as follows. In Section \ref{S:Pre}, some preliminaries on free semigroup algebras,
dual algebras and operator theory are given. Since the non-commutative analytic algebras play a very
crucial role in our paper, we take a closer look at them and prove two main results
in Section \ref{S:Ln}. Firstly, we prove that $(\L_m\ol\otimes \L_n)_*$ and ${\L_m}_*\widehat\otimes {\L_n}_*$ are completely
isomorphically isomorphic. This isomorphism will be used to prove that ${\L_{n}}_*$ is a completely contractive Banach algebra.
Secondly, it is shown that for any $k\ge 1$, the isometric $n$-tuple $[L_1^{\otimes k}\, \cdots \, L_n^{\otimes k}]$
induced from the left regular representation of $\Fn$
is analytic (actually, more can be said). We will apply this result to construct a Hopf algebra structure for an arbitrary analytic free semigroup $\fS$.
In Section \ref{S:AnaHop}, we show that
every analytic free semigroup algebra $\fS$ is a Hopf dual algebra, while in Section \ref{S:HCA}, we prove that its predual $\fS_*$
is a Hopf convolution dual algebra. Those structures between $\fS$ and $\fS_*$ are tied by completely bounded representations
of $\fS_*$. In Section \ref{S:cbrep}, we show that, on one hand,
there is a one-to-one correspondence between the set of completely bounded representations of $\fS_*$ and
the set of corepresentations of $\fS$; and
on the other hand, $\fS$ can be recovered from the coefficient operators of completely bounded representations
of $\fS_*$. Moreover,  it is shown that  the (Gelfand) spectrum of $\fS_*$ is $\Fn$ if $\fS$ is
generated by an analytic isometric $n$-tuple.

\section{Preliminaries}
\label{S:Pre}

In this section, we will give some background which will be used later. Also we take this opportunity to
fix our notation.

\subsection*{Free semigroup algebras}

The material in this subsection is mainly from \cite{Dav2} and the references therein.

A \textit{free semigroup algebra} $\fS$ is the unital $\wot$-closed operator algebra generated by
an isometric tuple.
Let us start with a distinguished representative of the class of those algebras --
the non-commutative analytic Toeplitz algebra $\L_n$.
Let $\Fn$ be the unital free semigroup
with the unit $\mt$ (the empty word), which is generated by the non-commutative symbols $1,...,n$. We form the Fock space
$\H_n:=\Fock$ with the orthonormal basis $\{\xi_w: w\in\Fn\}$. Consider the left regular representation $L$
of $\Fn$: For $i=1,...,n$,
\[
L_i(\xi_w)=\xi_{iw} \quad (w\in \Fn).
\]
It is easy to see that $L=[L_1\,\cdots\,L_n]$ is isometric.
Then the free semigroup algebra generated by $L$ is denoted as $\L_n$,
which is known as the \textit{non-commutative analytic Toeplitz algebra}.
The case of $n=1$ yields the $\wot$-closed algebra generated by the unilateral shift,
and so reproduces the (classical) analytic Toeplitz algebra.

Similarly, one can define the right regular representation $R$ of $\Fn$:  For $i=1,...,n$,
\[
R_i(\xi_w)=\xi_{wi}\quad (w\in \Fn).
\]
Then by $\R_n$ we mean the free semigroup algebra generated by the isometric $n$-tuple $R=[R_1\,\cdots\, R_n]$.
It is easy to see that $\L_n$ and $\R_n$ are unitarily equivalent.
Moreover, it turns out that $\L_n$ and $\R_n$ are the commutants of each other:
$\R_n=\L_n'$ and $\L_n=\R_n'$.

As shown in \cite{DavPit}, every element $A\in\L_n$ is uniquely determined by its ``Fourier expansion''
\[
A\sim\sum_{w\in\Fn}a_wL_w
\]
in the sense that
\[
A\xi_w=\sum_{w\in\Fn}a_w\xi_w\quad  (w\in\Fn),
\]
where the ``$w$-th Fourier coefficient'' is given by $a_w=(A\xi_\mt,\xi_w)$.
It is often useful to heuristically view $A$ as its Fourier expansion.
For $k\ge 1$, the $k$-th Ces\'aro sum of the Fourier series of $A$ is defined by
\[
\Sigma_k(A)=\sum_{|w|<k}(1-\frac{|w|}{k})a_wL_w,
\]
where $|w|$ is the length of the word $w$ in $\Fn$.
The sequence $\{\Sigma_k(A)\}_k$
converges to $A$ in the strong operator topology.

Another important operator algebra naturally associated with the left regular representation $L$ of $\Fn$ is
the \textit{non-commutative disk algebra} $\A_n$, which, by definition, is the unital norm closed algebra generated by $L$.
Clearly, $\A_n$ is weak*-dense in $\L_n$.
The case of $n=1$ yields the (classical) disk algebra.
The algebra $\A_n$ first appeared in \cite{Pop89}, which is related to multivariable non-commutative dilation theory.
It was shown in \cite{Pop96} that the unital norm closed operator algebra generated by an arbitrary isometric $n$-tuple
is completely isometrically isomorphic to $\A_n$.

Two classes of free semigroup algebras are particularly important and they are defined in terms of $\L_n$ and $\A_n$:

\begin{defn} (\cite{DavKatPit, DLP, Ken})
Let $\fS$ be a free semigroup algebra generated by an isometric $n$-tuple  $S$.
\begin{itemize}
\item[(i)] $\fS$ is said to be \textit{analytic} if it is algebraically isomorphic to $\L_n$.

\item[(ii)] $\fS$ is said to be \textit{absolutely continuous}
if the representation of $\A_n$ induced by $S$ can be extended to a weak*-continuous representation of
$\L_n$.
\end{itemize}

Sometimes, we also call the isometric tuple $S$ \textit{analytic} in (i) and \textit{absolute continuous} in (ii), respectively.
\end{defn}

In \cite{DavKatPit}, analytic free semigroup algebras are called \textit{type L}. The notion used here follows \cite{Ken}.
Analyticity plays a fundamental role in the existence of wandering vectors.
In \cite{DavKatPit} it was conjectured  that every
analytic free semigroup algebra has wandering vectors. Recently, this conjecture has been settled in \cite{Ken0}.
It is shown there that a free semigroup algebra either has a wandering vector, or is
a von Neumann algebra.
Motivated by the classical case, the notion of absolute continuity was introduced by Davidson-Li-Pitts in \cite{DLP},
and provided a major device there. By Theorem \ref{T:DKP} below and the definitions, analyticity implies absolute continuity.
It turns out that these two notions coincide in higher-dimensional cases \cite{Ken}.
However, this is not true in the classical case (cf. \cite{Ken, Wer}).

In what follows, we record some results of analytic free semigroup algebras for later reference.
Recall that a weak*-closed operator algebra $\A$ on $\H$ has \textit{property $\bA_1(1)$}, if
given a weak*-continuous linear functional $\tau$ on $\fS$ and $\epsilon>0$, there
are vectors $x,y\in \H$ such that $\|x\| \|y\|\le \|\tau\|+\epsilon$ and $\tau=[xy^*]$.
In particular, the weak* and weak topologies on such an algebra coincide.

\begin{thm} \rm{(\cite{DavKatPit})}
\label{T:DKP}
If $\fS$ is an analytic free semigroup algebra generated by an isometric $n$-tuple $S=[S_1\, \cdots \, S_n]$,
then there is a canonical weak*-homeomorphic and completely isometric isomorphism
$\phi$ from $\fS$ onto $\L_n$, which maps $S_i$ to $L_i$ for $i=1,...,n$.
\end{thm}

\begin{thm}
\label{T:ww} \rm{(\cite{Ber1} for $n=1$ and \cite{Ken} for $n\ge 2$)}
Any analytic free semigroup algebra $\fS$ has property $\bA_1(1)$.
\end{thm}

\subsection*{Dual algebras and operator spaces}

The main sources of this subsection are  \cite{ER, ER1, Ruan}.

In this paper,  a \textit{dual algebra} is a \textit{unital} weak*-closed subalgebra of $\B(\H)$ for some Hilbert space $\H$.
Usually, this is called a (unital) \textit{concrete dual algebra}.
There is a characterization for abstract dual algebras, and
it turns out that every abstract dual algebra has a concrete realization.
We will not need those facts in this paper.
Observe that free semigroup algebras are dual algebras.

Let $\A\subseteq\B(\H)$ be a dual algebra. Then $\A$ has the standard predual $\B(\H)_*/{\A_\perp}$, where,
as usual, $\B(\H)_*$ is the space of all bounded weak*-continuous
linear functionals on $\B(\H)$,  and $\A_\perp$ is the preannihilator of $\A$ in $\B(\H)_*$.
Of course, $\A$ may have more than one predual. We will use the notation ${\A_\H}_*$, or simply $\A_*$ if the context is clear, to denote
its standard predual.  This will not cause any confusion in the context of
free semigroup algebras because of the uniqueness of their preduals (cf. \cite{DavWri}).

It is very useful to notice that dual algebras are intimately connected with the theory of operator spaces.
Let $V$ be an operator space. Then its dual $V^*$ is
also an operator space, which is called the \textit{operator space dual} of $V$.
An operator space $W$ is said to be a \textit{dual operator space} if $W$ is completely isometrically isomorphic to $V^*$ for some operator space $V$. In this case, we also say that $V$ is an \textit{operator predual} of $W$.
The normal spatial tensor product of dual operator spaces (resp. algebras) is again a dual operator space (resp. algebra).
Any dual algebra $\A$ is a dual operator space.
Furthermore, $\A_*$ inherits a natural operator space structure
from $\A^*$, and so $\A_*$ itself is an operator space.
In fact, this paper heavily relies on the operator space structure of $\A_*$.

Given dual operator spaces $V^*$ and $W^*$ in $\B(\H)$ and $\B(\K)$, respectively,
the \textit{normal Fubini tensor product} $V^*\ol\otimes_\F W^*$ is defined by
\begin{align*}
V^*\ol\otimes_\F W^*=\{A\in\B(\H\otimes \K): \ & (\omega_1\otimes \id)(A)\in W^*, (\id\otimes \omega_2)(A)\in V^*\\
 &\mbox{for all}\ \omega_1\in \B(\H)_*, \ \omega_2\in\B(\K)_*\},
\end{align*}
where $\omega_1\otimes \id$ and $\id\otimes \omega_2$ are the right and left slice mappings determined
by $\omega_1$ and $\omega_2$, respectively.
Let $V^*\ol\otimes W^*$ and $V\widehat\otimes W$ stand for the normal spatial tensor product of $V^*$ and $W^*$,
and the operator projective tensor product of $V$ and $W$, respectively.
Since $\omega_1\otimes \id$ and $\id\otimes \omega_2$
are weak*-continuous, we have $V^*\ol\otimes W^*\subseteq V^*\ol\otimes_\F W^*$.
The following result characterizes when they are equal.

\begin{thm}
\label{T:ERR}
 \rm{(\cite{ER1, Ruan})}
(i) We have the following weak*-homeomorphic completely isometric isomorphism:
\[
V^*\ol\otimes_\F W^*\cong(V\widehat\otimes W)^*.
\]

(ii) $V^*\ol\otimes_\F W^* = V^*\ol\otimes W^*$ if and only if $(V^*\ol\otimes W^*)_*\cong V\widehat\otimes W$.
\end{thm}

\smallskip
\subsection*{Notation}
For a Hilbert space $\H$, we use $\H^k$ and $\H^{\otimes k}$ to denote the direct sum of $k$ copies of $\H$ and
 the tensor product of $k$ copies of $\H$, respectively. For $A\in\B(\H)$, by $A^{(k)}\in\B(\H^{k})$ and $A^{\otimes k}\in\B(\H^{\otimes k})$,
 we mean the direct sum of $k$ copies of $A$ acting on $\H^k$  and the tensor product of $k$ copies of $A$ acting on $\H^{\otimes k}$, respectively.

Let $S=[S_1\,\cdots\,S_n]$ be an arbitrary isometric $n$-tuple acting on $\H$.
We think of $S$ either as $n$ isometries acting on the common Hilbert space $\H$ or as an isometry from $\H^n$ to $\H$.
Also, for  $k\ge 1$, we set
 $S^{\otimes k}= [S_1^{\otimes k}\,\cdots \, S_n^{\otimes k}]$.

Let $\A$ and $\B$ be two dual algebras. We use
$\A\ol\otimes\B$ to denote the normal spatial tensor product of $\A$ and $\B$.

\section{$\L_n$ revisited}
\label{S:Ln}

In this section, we will take closer look at the non-commutative analytic Toeplitz algebras, and
prove some results which are vital in constructing Hopf algebra structures of an analytic free semigroup algebra
and its predual.

By $\L_m\ol\otimes^{\rm w} \L_n$, we mean the
$\wot$-closed algebra generated by spatial tensor product of $\L_m$ and $\L_n$.
The notation $\R_m\ol\otimes^{\rm w} \R_n$ is defined similarly.
It is easy to see that $\R_m\ol\otimes^{\rm{w}} \R_n\subseteq (\L_m\ol\otimes^{\rm{w}} \L_n)'$.
Then, by \cite[Theorem 4.3]{Berc}, $\L_m\ol\otimes^{\rm{w}} \L_n$
has property $\bA_1(1)$ if $m>1$ or $n>1$, so that the weak* and $\wot$ topologies on it coincide.
This is also the case if $m=n=1$ (cf., e.g., \cite{BerWes}).
Thus $\L_m\ol\otimes^{\rm{w}} \L_n=\L_m\ol\otimes\L_n$ for all $m,n\in\bN$. By \cite{KriPow},
$\L_m\ol\otimes\L_n$ can be identified with the higher rank analytic Toeplitz operator algebra $\L_\Lambda$
associated with a simple rank 2 graph $\Lambda$.
Using a general result on rank 2 graphs there, one has the following result (refer to \cite[Section 1 and Section 3]{KriPow}).

\begin{lem}
\label{L:com}
$
(\L_m\ol\otimes \L_n)'=\R_m\ol\otimes \R_n \quad\mbox{and}\quad (\L_m\ol\otimes \L_n)''=\L_m\ol\otimes \L_n.
$
\end{lem}

In what follows, we identify the standard predual of  $\L_m\ol\otimes \L_n$.

\begin{prop}
\label{P:pretenpro}
We have the following completely isometric isomorphism
\begin{align}
\label{E:lmln}
(\L_m\ol\otimes \L_n)_* & \cong {\L_m}_*\widehat\otimes {\L_n}_*.
\end{align}
\end{prop}

\begin{proof}
Although \cite[Theorem 7.2.4]{ER} handles the selfadjoint dual algebras (i.e., von Neumann algebras),
the proof of our proposition follows the same line there. For self-containedness, it is also included here.

Clearly,
it suffices to show that  $\L_m\ol\otimes_\F \L_n\subseteq \L_m \ol\otimes \L_n$.
By Lemma \ref{L:com}, we just need to check that each $T\in \L_m\ol\otimes_\F \L_n$ commutes with all operators in the
commutant $(\L_m \ol\otimes \L_n)'=\R_m\ol\otimes \R_n$.
Thus it is sufficient to show that
\[
T(I\otimes B)=(I\otimes B)T \quad (B\in\R_n)
\]
and
\[
T(A\otimes I)=(A\otimes I)T\quad (A\in\R_m).
\]

For the first identity, notice the following identification (\cite{ER})
\begin{align*}
\pi: \B(\H_m \otimes \H_n)&\cong \C\B(\B(\H_m)_*, \B(\H_n))\\
X&\mapsto \pi(X): \omega_1\mapsto (\omega_1\otimes \id)(X).
\end{align*}
So it suffices to show that
\[
(\omega_1\otimes \id)(T(I\otimes B))=(\omega_1\otimes \id)((I\otimes B)T)
\]
for all $w_1\in {\B(\H_m)}_*$, namely,
\[
(\omega_1\otimes \id)(T)B=B(\omega_1\otimes \id)(T).
\]
The above identity holds true since
$T\in \L_m\ol\otimes_\F \L_n$ implies that $(\omega_1\otimes \id)(T)\in \L_n$.

The second identity can be proved similarly.
\end{proof}

Before giving our next result, let us recall that, for $1\le k \le \infty$, an isometric $n$-tuple $S$
is said to be \textit{pure of multiplicity $k$} if $S$ is unitarily equivalent to $L^{(k)}$.
In particular, any pure isometric tuple is analytic.
Also, a subspace $\W$ is \textit{wandering for an isometric $n$-tuple $S$} if the subspaces $\{S_u\W : u\in \Fn\}$ are pairwise orthogonal.

\begin{prop}
\label{P:LiLi}
Let $L$ be the left regular representation of $\Fn$. Then
for any $k>1$, $L^{\otimes k}$ is a pure isometric $n$-tuple with multiplicity $\infty$.
\end{prop}

\begin{proof}\footnote{I am indebted to Adam Fuller for showing me this proof. }
Obviously, $L^{\otimes k}$ is an isometric $n$-tuple.
In what follows, we prove that $L^{\otimes k}$ is pure of multiplicity $\infty$.
Set
\[
\K=\ol\spn\{\xi_{u_1}\otimes\cdots\otimes \xi_{u_k}: u_1,...,u_n \mbox{ have no common prefix}\}.
\]
That is, $\K$ is spanned by
the basis vectors $\xi_{u_1}\otimes\cdots\otimes \xi_{u_k}$, where
$u_i=\mt$ for \textit{some} $i\in\{1,...,k\}$,
or $u_i=k_iu_i'$ with $k_i\ne k_j$ for \textit{some} $i\ne j$.

It is easy to check that $\K$ is wandering for $L^{\otimes k}$. Suppose now $\xi_{u_1}\otimes\cdots\otimes \xi_{u_k}$
is an arbitrary basis vector for $\H_n^{\otimes k}$ with $\xi_{u_1}\otimes\cdots\otimes \xi_{u_k}\not\in \K$.
One can choose a word $w\in \Fn$ of maximal
length such that $u_i=wu_i'$ for some $u_i'\in\Fn$ $(i=1,...,k)$. Notice that $w\neq \varnothing$ as $\xi_{u_1}\otimes\cdots\otimes \xi_{u_k}\not\in \K$.
On the other hand, since $w$ is maximal, it follows that $\xi_{u_1}'\otimes\cdots\otimes \xi_{u_k}'\in \K$. Clearly,
\[
\xi_{u_1}\otimes\cdots\otimes \xi_{u_k}=L_w^{\otimes k}(\xi_{u_1}'\otimes\cdots\otimes \xi_{u_k}').
\]
Thus, for any basis vector $\xi_{u_1}\otimes\cdots\otimes \xi_{u_k}\in \H_n^{\otimes k}$,
we have
\[
\xi_{u_1}\otimes\cdots\otimes \xi_{u_k}\in\bigoplus_{w\in\Fn}L_w^{\otimes k}\, \K,
\]
and so
\[
\H_n^{\otimes k}=\bigoplus_{w\in\Fn} L_w^{\otimes k}\,\K.
\]

Obviously, $\dim\K=\infty$ as $k>1$. Therefore, it follows from \cite{Pop89} that $L^{\otimes k}$ is a pure isometric tuple with multiplicity $\infty$.
\end{proof}

As an immediately consequence of the above proposition, we have that $L^{\otimes k}$ is analytic for all $k\in\bN$.

\section{Analytic free semigroup algebras are Hopf algebras}
\label{S:AnaHop}

Following Effros-Ruan in \cite{ER1}, we define the term Hopf algebras from analysts' point of view as follows.
A \textit{Hopf algebra} $(\A, m, \Delta)$ consists of a linear
space $\A$ with norms or matrix norms, an associative bilinear multiplication $m: \A\times \A\to \A$,
and a coassociative comultiplication $\Delta: A\to \A\widetilde{\otimes} \A$, where
$\widetilde\otimes$ is a suitable tensor product, and $\Delta$ is an algebra homomorphism. The maps
are assumed to be bounded in some appropriate sense.

In this paper, we are interested in two classes of Hopf algebras -- Hopf dual algebras and Hopf convolution dual algebras, where
the latter is induced from the former. The former is investigated in this section, and the latter will
be studied in next section. Some definitions are first.

\begin{defn}
\label{D:coalg}
(i) A dual algebra $\A$ is said to be a \textit{Hopf dual algebra} if there is an injective unital
weak*-continuous completely contractive
homomorphism $\Delta:\A\to \A\ol\otimes \A$
such that it is also coassociative:
\[
(\id\otimes \Delta)\Delta=(\Delta\otimes \id)\Delta.
\]
That is, the following diagram commutes:
\[
\begin{xy}
(0,20)*+{{\A}}="a"; (35,20)*+{{\A\ol\otimes\A}}="b";%
(0,0)*+{{\A}\ol\otimes{\A}}="c"; (35,0)*+{\A\ol\otimes\A\ol\otimes\A.}="d";%
{\ar@{->} "a";"b"};?*!/_2mm/{\Delta};
{\ar@{->} "a";"c"};?*!/^/{\Delta};
{\ar@{->} "b";"d"};?*!/_7mm/{\Delta\otimes\id};
{\ar@{->} "c";"d"};?*!/^2mm/{\id\otimes\Delta},
\end{xy}
\]

The homomorphism $\Delta$ is called a \textit{comultiplication}
or \textit{coproduct} on $\A$.

(ii) A Hopf dual algebra is \textit{counital} if there is a non-zero weak*-continuous homomorphism $\epsilon:\A\to\bC$ such that
\[
(\epsilon\otimes\id)\circ\Delta=\id=(\id\otimes\epsilon)\circ\Delta.
\]
The homomorphism $\epsilon$ is
called a \textit{counit} of $\A$.

(iii) A Hopf dual algebra $\A$ is said to be \textit{cocommutative}  if
$\sigma\circ \Delta=\Delta$, where $\sigma$ is the flip map $a\otimes b\mapsto b\otimes a$.
\end{defn}

We shall use the pair $(\A, \Delta)$, or $(\A, \Delta_\A)$ if $\A$ needs to be stressed, or simply $\A$ if the context is clear,
to denote the Hopf dual algebra $\A$ with the comultiplication $\Delta$.

\begin{defn}
\label{D:mor}
A \textit{morphism of two Hopf dual algebras} $(\A, \Delta_\A)$ and $(\B, \Delta_\B)$ is a unital weak*-continuous
completely contractive homomorphism $\pi:\A\to \B$ that makes $\Delta_\A$ and $\Delta_\B$ compatible in the following sense:
\[
\Delta_\B\circ \pi=(\pi \otimes \pi)\circ \Delta_\A.
\]

If $\pi$ is weak*-weak* homeomorphic and completely isometrical isomorphic, then $\A$ and $\B$ are said to be \textit{isomorphic as dual Hopf algebras}, and denoted
as $(\A,\Delta_\A)\cong(\B,\Delta_\B)$.
\end{defn}

\begin{defn}
\label{D:int}
Let $(\A, \Delta)$ be a Hopf dual algebra. A unital weak*-continuous linear functional $\varphi:\A\to \bC$ is
called a \textit{left} (resp. \textit{right}) \textit{integral} on $(\A,\Delta)$ if it is
left-invariant (resp. right-invariant), that is,
\[
(\id\otimes \varphi)(\Delta(a))=\varphi(a)I\quad (\mbox{resp.  }(\varphi\otimes \id)(\Delta(a))=\varphi(a)I).
\]

A left and right integral is briefly called an \textit{integral}.
\end{defn}

Some remarks for the above definitions are in order.

\begin{rem}
\label{R:sim1}
Note that if $\A$ in Definition \ref{D:coalg} is selfadjoint (i.e., $\A$ is a von Neumann algebra), then $(\A, \Delta)$ is nothing but  a Hopf
von Neumann algebra (\cite{Enock}). This is due to the fact that a unital completely contractive homomorphism between two
C*-algebras is automatically a *-homomorphism (\cite{Pau}).
So, as one expects, the notions of Hopf dual and Hopf von Neumann algebras coincide in the selfadjoint case.
Then, in this case, the notion of morphisms in Definition \ref{D:mor} is the same as the one for Hopf von Neumann algebras.
\end{rem}

\begin{rem}
\label{R:sim3}
All tensor maps involved in Definitions \ref{D:coalg}, \ref{D:mor} and \ref{D:int}
have unique weak*-continuous extensions. For example, $\Delta\otimes\id$ has a unique weak*-continuous extension,
still denoted by $\Delta\otimes\id$, which maps $\A\ol\otimes\A$ to $\A\ol\otimes\A\ol\otimes\A$.
This can be seen, e.g.,  from \cite{ER1} (also cf. the proof of Theorem \ref{T:abscon} below).
\end{rem}

\begin{rem}
\label{R:sim2}
Notice that integrals  in Definition \ref{D:int} are normalized.
It is not hard to check that integrals of a Hopf dual algebra are unique.
If $\A$ is a von Neumann algebra, then an integral is nothing but a Haar state.
\end{rem}

We are now ready to give the first main theorem in this section. It seems to have been overlooked in the classical case.

\begin{thm}
\label{T:coalg}
$(\L_n,\Delta)$ is a cocommutative Hopf dual algebra with the integral $\varphi_0=[\xi_\mt\xi_\mt^*]$,
where $\Delta$ is determined by $L_i\mapsto L_i\otimes L_i$ for $i=1,...,n$.
\end{thm}

\begin{proof}
Since $L^{\otimes 2}$ is an isometric $n$-tuple, 
there is a representation
$\Delta$ of the non-commutative disk algebra $\A_n$ determined by
\[
\Delta: \A_n\to \L_n\ol\otimes \L_n, \quad L_i\mapsto L_i\otimes L_i \quad (i=1,...,n).
\]
By Proposition \ref{P:LiLi}, $L^{\otimes 2}$ is analytic and so absolutely continuous. Thus
$\Delta$ can be extended to a weak*-continuous representation of $\L_n$, which is still denoted as $\Delta$ by
abusing notation:
\[
\Delta: \L_n\to \L_n\ol\otimes \L_n, \quad L_i\mapsto L_i\otimes L_i\quad (i=1,...,n).
\]

Let $A\sim\sum_wa_wL_w$ be an arbitrary element in $\L_n$. Then
\begin{align}
\nonumber
(\Delta\otimes\id)\circ\Delta(A)
&=\Delta\otimes\id(\Delta(\sot\!\!-\!\!\lim_k(\Sigma_k(A))))\\
\nonumber
&=\Delta\otimes\id(\Delta(\wot\!\!-\!\!\lim_k(\Sigma_k(A))))\\
\nonumber
&=\Delta\otimes\id(\Delta(\mbox{weak*}\!\!-\!\!\lim_k(\Sigma_k(A))))\\
\nonumber
&=\mbox{weak*}\!\!-\!\!\lim_k\Delta\otimes\id(\Delta(\Sigma_k(A)))\\
\nonumber
&=\mbox{weak*}\!\!-\!\!\lim_k\sum_{|w|<k}(1-\frac{|w|}{k})a_w\Delta\otimes\id(L_w\otimes L_w)\\
\nonumber
&=\mbox{weak*}\!\!-\!\!\lim_k\sum_{|w|<k}(1-\frac{|w|}{k})a_w(L_w\otimes L_w\otimes L_w)\\
\label{E:oper}
&=\wot\!\!-\!\!\lim_k\sum_{|w|<k}(1-\frac{|w|}{k})a_w(L_w\otimes L_w\otimes L_w),
\end{align}
where the above third ``='' uses the fact that $\wot$ and weak* topologies coincide on $\L_n$ by Theorem \ref{T:ww},
the fourth ``='' is from the weak*-continuity of $\Delta\otimes \id $ and $\Delta$,
and the last second one comes from the definition of $\Delta$.
Some obviously slight changes of the above proof yield
\[
(\Delta\otimes\id)\circ\Delta(A)
=\wot\!\!-\!\!\lim_k\sum_{|w|<k}(1-\frac{|w|}{k})a_w(L_w\otimes L_w\otimes L_w).
\]
Thus $\Delta$ is coassociative.

Similarly, one can show the cocommutativity of $\L_n$.
Thus $\L_n$ is a cocommutative Hopf dual algebra.

Finally, we verify that $\varphi_0=[\xi_\mt \xi_\mt^*]$ is the integral of $\L_n$.
Clearly $\varphi_0$ is a weak*-continuous linear functional and $\varphi_0(I)=1$.
Then using some calculations similar to the above, one obtains
\[
(\varphi_0\otimes \id)\circ \Delta(A)=\varphi_0(A)I=(\id\otimes \varphi_0)\circ \Delta(A)
\]
for all $A\in \L_n$. Therefore, $\varphi_0$ is the integral on $\L_n$.
\end{proof}

\begin{rem}
(i) By Proposition \ref{P:LiLi}, $L^{\otimes 3}$ is an analytic isometric $n$-tuple.
From the above proof, one actually obtains that the map
$(\Delta\otimes\id)\circ\Delta$, or $(\id\otimes\Delta)\circ\Delta$,
is nothing but the representation of $\L_n$ extended from that of $\A_n$ induced by $L^{\otimes 3}$:
\begin{align*}
(\Delta\otimes\id)\circ\Delta=(\id\otimes\Delta)\circ\Delta: \L_n&\to \L_n\ol\otimes \L_n\ol\otimes\L_n\\
L_i&\mapsto L_i^{\otimes 3}\quad (i=1,...,n).
\end{align*}
Also, the operator in \eqref{E:oper}
is the one with the Fourier expansion
$\sum_w a_w(L_w\otimes L_w\otimes L_w)$.
(Refer to the proof of Theorem \ref{T:char} below for some related details.)

(ii) It is worth noticing that, in the free group case with $n\ge 2$, the above integral $\varphi_0$ is the canonical faithful trace of the
group von Neumann algebra $\vn(\bF_n)$, which is a II$_1$ factor.
\end{rem}

\medskip
One can now generalize the above result  to an arbitrary analytic isometric tuple.

\begin{thm}
\label{T:abscon}
Suppose that $\fS$ is an analytic free semigroup algebra generated by an isometric $n$-tuple $S=[S_1\, \cdots \, S_n]$. Then
$(\fS,\Delta_\fS)$ is a cocommutative Hopf dual algebra with an integral,
where $\Delta_\fS$ maps $S_i$ to $S_i\otimes S_i$ for all $1\le i\le n$.
\end{thm}

\begin{proof}
Let $\phi:\fS\to \L_n$ be the canonical weak*-homeomorphic completely isometric isomorphism.
From \cite[Corollary 4.1.9]{ER}, we have that
$\phi_*:{\L_n}_*\to \fS_*$ is completely isometrically isomorphic, and so is
$\phi_*\otimes \phi_*: {\L_n}_*\widehat\otimes {\L_n}_*\to \fS_*\widehat\otimes\fS_*$
by \cite[Proposition 7.1.7]{ER}.
Applying  \cite[Corollary 4.1.9]{ER}  again gives that
$(\phi_*\otimes \phi_*)^*:(\fS_*\widehat\otimes\fS_*)^*\to ({\L_n}_*\widehat\otimes {\L_n}_*)^*$
is a weak*-homeomorphic completely isometric isomorphism.
Hence $(\phi_*\otimes \phi_*)^*$ is a weak*-homeomorphic completely isometric isomorphism
between $\fS\ol\otimes_\F\fS$ and $\L_n\ol\otimes_\F\L_n$ by Theorem \ref{T:ERR}.
But  it follows from Theorem \ref{T:ERR} and Proposition \ref{P:pretenpro} that $\L_n\ol\otimes\L_n=\L_n\ol\otimes_\F\L_n$.
Hence one has
\begin{align}
\label{E:SS}
\fS\ol\otimes\fS=\fS\ol\otimes_\F\fS.
\end{align}
Let $\Phi$ be the inverse of $(\phi_*\otimes \phi_*)^*$ (which is actually the extension of $\phi^{-1}\otimes\phi^{-1}$).
Then $\Phi$ is a weak*-homeomorphic completely isomorphism from $\L_n\ol\otimes\L_n$ onto $\fS\ol\otimes\fS$.

Set $\Delta_\fS:=\Phi\circ\Delta\circ\phi$, where $\Delta$ is the comultiplication on $\L_n$ given in Theorem \ref{T:coalg}.
By \cite{DavYang}, $\Delta$ is weak*-continuous completely isometrically homomorphic.
Then $\Delta_\fS: \fS\to \fS\ol\otimes\fS$ is a unital weak*-continuous completely isometric homomorphism. Also,
from the above analysis, $\Delta_\fS$ maps $S_i$ to $S_i\otimes S_i$
($i=1,..., n$). Then, as in the proof of Theorem \ref{T:coalg},  it is not hard to check the coassociativity of $\Delta_\fS$
and the cocommutativity of $\fS$. Therefore, $\fS$ is a cocommutative Hopf dual algebra.
Let $\epsilon=\varphi_0\circ \phi$. Then one can easily verify that $\epsilon$ is
an integral of $\fS$.
\end{proof}

As a byproduct of the above proof, one can construct analytic isometric tuples from a given one as follows.

\begin{cor}
\label{C:S2abs}
Suppose that $S$  is an analytic isometric tuple.
Then so is $S^{\otimes k}$ for any $k\ge 1$.
\end{cor}

\begin{proof}
Let $k\ge 1$ and $S$ be an analytic isometric $n$-tuple.
As obtaining $\Phi$ in the proof of Theorem \ref{T:abscon}, one has that
there is a weak*-homeomorphic completely isometric isomorphism
$\widetilde \Phi$ between $\ol\otimes_{i=1}^k \L_n$ and
$\ol\otimes_{i=1}^k\fS$, the normal spatial tensor product of $k$ copies of $\fS$.
Let $\pi: \L_n\to \ol\otimes_{i=1}^k \L_n$ be the weak*-continuous, completely isometric homomorphism determined by Proposition \ref{P:LiLi}.
Then the composition $\widetilde\Phi\circ\pi$ is the weak*-continuous extension of the representation of $\A_n$ induced by $S^{\otimes k}$.
Therefore, $S^{\otimes k}$ is analytic.
\end{proof}

Suppose that $S$  is an analytic isometric $n$-tuple. Then ``$\L_n = \fS$'' as operator algebras.
As before, one can check that the canonical homomorphism $\sigma: \fS\to \L_n$ makes the comultiplications $\Delta$ of $\L_n$ and $\Delta_\fS$
of $\fS$ compatible: $\Delta\circ\phi=(\phi\otimes \phi)\circ \Delta_\fS$. Thus, ``$\L_n = \fS$''  as Hopf dual algebras.
We record this as a corollary.

\begin{cor}
\label{C:mor}
Suppose that $\fS$ is an analytic free semigroup algebra generated by an isometric $n$-tuple. Then
$(\L_n,\Delta)\cong (\fS,\Delta_\fS)$.
\end{cor}

\section{Hopf Convolution Dual Algebras}
\label{S:HCA}

Let $\fS$ be an analytic free semigroup algebra generated by an isometric $n$-tuple. In this section,
we shall show that its predual $\fS_*$ is also a Hopf algebra.
Naturally, $\fS_*$ and ${\L_n}_*$ are canonically isomorphic as Hopf algebras.
Analogous to the Fourier algebra $A(\bF_n)$, the algebra ${\L_n}_*$ is non-unital.

In order to equip $\fS_*$  with an algebra structure,
let $m_{\fS_*}=(\Delta_\fS)_*$, where $\Delta_\fS$ is the comultiplication of $\fS$
constructed in the previous section. Then it follows essentially from Proposition \ref{P:pretenpro} and Theorem \ref{T:abscon}
 that $m_{\fS_*}$ gives a completely contractive
multiplication of $\fS_*$. Thus $\fS_*$ is a completely contractive Banach algebra.

To construct a coalgebra structure of $\fS_*$, we follow the same line with Effros-Ruan \cite[Section 7]{ER1}.
The idea is sketched as follows. In order to endow $\fS_*$ with  a comultiplication $\Delta_{\fS_*}$,
one needs to use the normal Haagerup tensor product
$\fS\stackrel{\sigma\rm h}\otimes  \fS$, because it is suitable for linearizing the multiplication $m$ on $\fS$ in the following sense:
The multiplication $m: \fS\times\fS\to\fS$  extends uniquely to a weak*-continuous
\textit{completely contractive} map $m_\fS: \fS\stackrel{\rm{\sigma}h}\otimes\fS\to \fS$.
Then it turns out that its preadjoint $(m_{\fS})_*$ produces a comultiplication of $\fS_*$.
To prove that it is indeed an \textit{algebra homomorphism}, we need two important ingredients:
(i) a Shuffle Theorem,
and (ii) $(\fS\ol\otimes\fS)_*\cong \fS_*\widehat\otimes\fS_*$.
Fortunately, we have both (i) and (ii):  (i) holds true ``automatically'' because the Shuffle Theorem \cite[Theorem 6.1]{ER1}
holds true for
\textit{arbitrary} operator spaces,
and (ii) can be proved based on some results obtained in previous sections.

Recall that, given operator spaces $V$ and $W$, the normal and extended Haagerup tensor products are connected by
$V^*\stackrel{\sigma\rm h}\otimes  W^*=(V\stackrel{\rm{eh}}\otimes  W)^*$.
Also, a complex algebra $\A$ is called a \textit{completely contractive Banach algebra}
if $\A$ is an operator space and the multiplication $m:\A\times\A\to\A$
is a completely contractive bilinear mapping, namely, it determines a completely contractive linear mapping $m:\A\widehat\otimes\A\to\A$.

\begin{thm}
\label{T:Conv}
Suppose that $\fS$ is an analytic free semigroup algebra. Then
$(\fS_*, m_{\fS_*}, \Delta_{\fS_*})$ is a Hopf algebra, where
the multiplication $m_{\fS_*}: \fS_*\widehat\otimes\fS_*\to \fS_*$
and the comultiplication $\Delta_{\fS_*}: \fS_*\to \fS_*\stackrel{\rm{eh}}\otimes\fS_*$
are completely contractive, and they are given by
$m_{\fS_*}=(\Delta_\fS)_*$
and
$\Delta_{\fS_*}=(m_\fS)_*$, respectively.
\end{thm}

\begin{proof}
By Proposition \ref{P:pretenpro} and the identity \eqref{E:SS}, one has
\begin{align}
\label{E:Stenpro}
(\fS\ol\otimes\fS)_* \cong \fS_*\widehat\otimes\fS_*.
\end{align}
Since the comultiplication $\Delta_\fS: \fS\to \fS\ol\otimes\fS$ is
a weak*-continuous complete isometry (from the proof of Theorem \ref{T:abscon}),  it follows from \cite[Corollary 4.1.9]{ER} that
the preadjoint $(\Delta_\fS)_*$
is a completely quotient mapping from ${\fS}_*\widehat\otimes {\fS}_*$ to
${\fS}_*$. Hence $(\fS_*, (\Delta_\fS)_*)$ is a completely contractive Banach algebra.
More precisely, for any $\varphi, \psi\in{\fS}_*$, the multiplication
$\varphi\ast\psi:=(\Delta_\fS)_*(\varphi\otimes \psi)$ is defined by
\[
\varphi\ast\psi(x)=\varphi\otimes \psi(\Delta_\fS(x)) \quad (x\in \fS).
\]
It is now straightforward to check that $\varphi\ast\psi=\psi\ast\varphi$, implying the commutativity of $(\fS_*,\ast)$.
Therefore $({\fS}_*, (\Delta_\fS)_*)$ is a
completely contractive commutative Banach algebra.

Using the isomorphism
\eqref{E:Stenpro}
and the Shuffle Theorem \cite[Theorem 6.1]{ER1},
one can obtain that $(m_\fS)_*$ is a completely contractive (algebra) homomorphism.
The proof is completely similar to that of \cite[Theorem 7.1]{ER1} (also refer to the discussion preceding the statement of this theorem),
and so it is omitted here.

Therefore $(\fS_*,(\Delta_\fS)_*, (m_\fS)_*)$ is a Hopf algebra.
\end{proof}

Applying Theorems \ref{T:coalg} and \ref{T:Conv} to the case of $n=1$, we immediately obtain the following result in the classical case,
which seems new.  Recall that
$(H^\infty)_*=L_1(\bT)/H_1(0)$, where $H_1(0)=zH_1(\bT)$ (\cite{Koo}).

\begin{cor}
$H^\infty$ and $L_1(\bT)/H_1(0)$ are Hopf algebras.
\end{cor}



Following Effros-Ruan \cite{ER1}, Hopf algebras obtained from Theorem \ref{T:Conv} are called \textit{Hopf convolution dual algebras}.
Of course, counits, morphisms and integrals can be defined as those in Definitions \ref{D:coalg}, \ref{D:mor} and \ref{D:int},
but without the weak*-continuity requirement.

\medskip

Since $\bF_n$ is a free group, it is well known that the Fourier algebra $A(\bF_n)$ is not unital.
Actually, for $n\ge 2$, more is true: $A(\bF_n)$ does not have bounded approximate identities.
In our case, so far we can show that ${\L_n}_*$ is not unital. Before proving this result, let us first recall from
\cite{Dav2, DavPit} that the set $M({\L_n})$ of
$\wot$-continuous multiplicative linear functionals consists of those linear functionals having the form
$\varphi_\lambda=[\nu_\lambda\nu_\lambda^*]$, where $\lambda\in\bB_n$ (the open unit ball of $\bC^n$) and
$\nu_\lambda$ is a certain unit vector in $\H_n$ (whose precise definition is not important for us and so is omitted here).
It is known that $\varphi_\lambda(p(L))=p(\lambda)$ for any polynomial
$p=\sum_wa_ww$ in the semigroup algebra $\bC\Fn$.

\begin{prop}
\label{P:nonuni}
If $\fS$ is an analytic free semigroup algebra, then $\fS_*$ is non-unital.
\end{prop}

\begin{proof}
It is not hard to see that it suffices to show that
$\fS$ is not counital.
To the contrary, assume that it has a counit $\epsilon_\fS$. Then
$\epsilon:=\phi^{-1}\circ \epsilon_\fS$ is a counit of $\L_n$. So $\epsilon=\varphi_\lambda$ for some $\lambda\in\bB_n$.
From the identity required in the definition of a counit, it is easy to check that $\epsilon(L_w)=1$ for all $w\in\Fn$. However,
as mentioned above,
$\varphi_\lambda(L_w)=w(\lambda)$,
where $w(\lambda)=\lambda_1^{k_1}\cdots\lambda_n^{k_n}$ for
$w=i_1^{k_1}\cdots i_n^{k_n}$. This particularly forces $\lambda_i=1$ for all $i=1,...,n$.
Obviously, this is impossible as $\lambda\in\bB_n$.
\end{proof}

It is worthy to notice that $M({\L_n})$ is also closed under the multiplication $\ast$ of ${\L_n}_*$
and has an involution $\dag$.
In fact, for $\varphi_\lambda,\ \varphi_\mu\in \{\varphi_\nu: \nu\in\bB_n\}$, then
$\varphi_\lambda \ast \varphi_\mu=\varphi_{\lambda*\mu}$, where $\lambda*\mu=(\lambda_1\mu_1,...,\lambda_n\mu_n)$
for $\lambda=(\lambda_1,...,\lambda_n)$ and $\mu=(\mu_1,...,\mu_n)$ in $\bB_n$.
Also, $\dag$ can be defined by $\varphi_\lambda^\dag=\varphi_{\ol\lambda}$.

\medskip

If $\fS$ is an analytic free semigroup algebra generated by an isometric $n$-tuple,
then $\fS$ and $\L_n$ are completely isometrically isomorphic. So $\fS_*\cong {\L_n}_*$ as operator spaces.
The following result tells us this is also the case as Hopf convolution algebras.

\begin{prop}
Suppose that $\fS$ is an analytic free semigroup algebra generated by an isometric $n$-tuple.
Then $\fS_*\cong {\L_n}_*$ as Hopf convolution algebras.
\end{prop}

\begin{proof}
For brevity, we use $*$ for the multiplication of ${\L_n}_*$ and that of $\fS_*$ constructed above.
It follows from \cite[Corollary 4.1.9]{ER} that $\phi_*:{\L_n}_*\to \fS_*$ is completely isometric.
We now show that it is also an algebra homomorphism,
namely, $\phi_*(f\ast g)=\phi_*(f) \ast \phi_*(g)$ for all $f,g\in{\L_n}_*$.
Indeed, we have
\begin{align*}
\phi_*(f\ast g)
&=(f\ast g)\circ\phi=(f\otimes g)\circ\Delta\circ \phi\\
&=(f\otimes g)\circ\Phi\circ\Delta_\fS \\
&=(f\otimes g)\circ(\phi\otimes\phi)\circ\Delta_\fS \\
&=(\phi_*\otimes \phi_*)\circ(f\otimes g)\circ\Delta_\fS \\
&=(\phi_*(f)\otimes \phi_*(g))\circ\Delta_\fS \\
&=\phi_*(f)\ast\phi_*(g).
\end{align*}
Also, it is easy to check that $(\phi_*)^{-1}=(\phi^{-1})_*$.
Therefore, $\phi_*$ is an algebra isomorphism.

In order to show that $\phi_*$ is compatible with the comultiplications on ${\L_n}_*$ and $\fS_*$,
consider the following commuting diagram
\[
\begin{xy}
(0,20)*+{{\fS}}="a"; (40,20)*+{{\L_n}}="b";%
(0,0)*+{{\fS}\stackrel{\rm{\sigma h}}\otimes{\fS}}="c"; (40,0)*+{{\L_n}\stackrel{\rm{\sigma h}}\otimes{\L_n}}="d";%
{\ar@{->} "a";"b"};?*!/_2mm/{\phi};
{\ar@{->} "c";"a"};?*!/_5mm/{m_{\fS}};
{\ar@{->} "d";"b"};?*!/^5mm/{m_{\L_n}};
{\ar@{->} "c";"d"};?*!/^2mm/{\phi\otimes\phi},
\end{xy}
\]
and then take its preadjoint
\[
\begin{xy}
(0,20)*+{{\fS}_*}="a"; (40,20)*+{{\L_n}_*}="b";%
(0,0)*+{{\fS}_*\stackrel{\rm{eh}}\otimes{\fS}_*}="c"; (40,0)*+{{\L_n}_*\stackrel{\rm{eh}}\otimes{{\L_n}_*}}="d";%
{\ar@{<-} "a";"b"}?*!/_2mm/{\phi_*};
{\ar "a";"c"};{\ar "b";"d"};%
{\ar "a";"c"};?*!/^5mm/{\Delta_{{\fS_*}}};
{\ar@{->} "b";"d"};?*!/_5mm/{\Delta_{{L_n}_*}};
{\ar@{<-} "c";"d"};?*!/^4mm/{\phi_*\stackrel{\rm{eh}}\otimes\phi_*};
\end{xy}.
\]
\end{proof}

\section{Completely Bounded Representations of $\fS_*$} 
\label{S:cbrep}

Let $\fS$ be an analytic free semigroup algebra.
In this section, we shall study completely bounded representations of $\fS_*$.
Through those representations, $\fS$ and $\fS_*$ are closely connected with each other.
On one hand, we show that  there is a bijection between
the set of completely bounded representations of $\fS_*$ and the set of corepresentations of $\fS$. On the other hand,
$\fS$ can be recovered by completely bounded representations of $\fS_*$:
The algebra induced from the coefficient operators of completely bounded representations
of $\fS_*$ is precisely $\fS$. This gives us a new duality between $\fS_*$ and $\fS$.

As an amusing application of the former bijection, we prove that the spectrum of $\fS_*$, as a
commutative Banach algebra, is precisely $\Fn$.
In order to achieve the latter duality, one needs to
introduce the notion of tensor products of completely bounded representations of $\fS_*$.
Fortunately, for this one can borrow it from \cite[Section 7]{ER1}. Actually, the entire subsection on tensor products
is from there. It should be mentioned that we have the equality $\C(\fS_*)=\fS$ as Theorem \ref{T:dual},
rather than just an inclusion $\C(\fS_*)\subseteq\fS$ in \cite[Theorem 7.2]{ER1}, because we are dealing with only
 ``Fourier algebras'', rather than general Hopf convolution dual algebras.

\subsection*{Corepresentations of $\fS$}

We begin this subsection with the following definition.

\begin{defn}
\label{D:corep}
Let $(\A, \Delta)$ be a Hopf dual algebra acting on $\H$.
A \textit{corepresentation} of  $(\A,\Delta)$ on a Hilbert space $\K$ is an (arbitrary) operator $V\in \A\otl\B(\K)$ satisfying
\[
V_{1,3}V_{2,3}= (\Delta\otimes \id)(V).
\]
Here $V_{1,3}$ and $V_{2,3}$ are the standard \textit{leg notation} (cf. \cite{BS}):  $V_{1,3}$ is a linear operator on the Hilbert space
$\H\otimes\H\otimes\K$, which acts as $V$ on the first and third tensor factors and as the identity on the second one.
The notation $V_{2,3}$ is defined similarly.
\end{defn}

Before stating our first main theorem in this section, we need the following lemma, which
is probably a folklore. Here we give a direct proof based on special properties of
$\L_n$.

\begin{lem}
\label{L:LnBK}
Suppose that $\fS$ is an analytic free semigroup algebra. Then
we have a completely isometric isomorphism
\[
\C\B({\fS}_*, \B(\K)) \cong \fS \ol\otimes \B(\K).
\]
\end{lem}

\begin{proof}
We first claim that this lemma holds true when $\fS=\L_n$:
\begin{align}
\label{E:Lncor}
\C\B({\L_n}_*, \B(\K))\cong\L_n\ol\otimes \B(\K).
\end{align}

To this end, observe that
\[
(\L_n\ol\otimes \B(\K))'=\R_n\ol\otimes \bC I
\quad \mbox{and}\quad
(\L_n\ol\otimes \B(\K))''=\L_n\ol\otimes \B(\K).
\]
Then, applying some obvious modifications to the proof of Proposition \ref{P:pretenpro}, one can show that
$({\L_n}_*\widehat\otimes \B(\K)_*)^*$ and  $\L_n\ol\otimes \B(\K)$
are completely isometrically isomorphic.
On the other hand, there is a natural completely isometric isomorphism
$({\L_n}_*\widehat\otimes \B(\K)_*)^*\cong\C\B({\L_n}_*, \B(\K))$ by \cite[Corollary 7.1.5]{ER}.
This proves our claim.

Now suppose $\fS$ is an arbitrary analytic free semigroup algebra. Using a proof completely similar
to that of Theorem \ref{T:abscon}, we obtain that
\[
({\fS}_*\widehat\otimes \B(\K)_*)^*\cong({\L_n}_*\widehat\otimes \B(\K)_*)^*
\]
and
\[
\fS\ol\otimes \B(\K) \cong \L_n\ol\otimes \B(\K).
\]
Therefore
\begin{align*}
\C\B({\fS}_*, \B(\K))
& \cong({\fS}_*\widehat\otimes \B(\K)_*)^*\cong({\L_n}_*\widehat\otimes \B(\K)_*)^*\\
& \cong\C\B({\L_n}_*, \B(\K))\cong\L_n\ol\otimes \B(\K) \\
& \cong \fS \ol\otimes \B(\K),
\end{align*}
where the first $\cong$ is from \cite[Corollary 7.1.5]{ER}, while the fourth one is from \eqref{E:Lncor}.
\end{proof}

\begin{thm}
\label{T:corep}
Suppose that $\fS$ is an analytic free semigroup algebra acting on a Hilbert space $\H$. Then
there is a bijection between the set of all corepresentations of $\fS$
and the set of all completely bounded representations
of $\fS_*$.

More precisely, if $V$ is a corepresentation of $\fS$ on $\K$, then the corresponding representation $\pi_V$ of ${\fS}_*$ is given by
\begin{align}
\label{E:piV}
\pi_V(\varphi)=(\varphi\otimes \id)(V) \quad (\varphi\in\fS_*);
\end{align}
if $\pi$ is a completely bounded representation of $\fS_*$ on $\K$, then the
corresponding corepresentation $V_\pi$ of $\fS$ is given by
\begin{align}
\label{E:Vpi}
(V_\pi(\xi\otimes x), \eta\otimes y)=(\pi([\xi\eta^*]) x,y)
\end{align}
for all $\xi,\eta\in\H$, $x,y\in\K$.
\end{thm}

\begin{proof}
The following calculations show that, for a corepresentation $V$ of $\fS$, the map $\pi_V$ defined in \eqref{E:piV}
indeed yields a representation of $\fS_*$ on $\K$: For $\varphi\in\fS_*$, we have
\begin{align*}
(\varphi\otimes \id)(V)(\psi\otimes \id)(V)
&=(\varphi\otimes\psi\otimes \id)(V_{1,3}V_{2,3})\\
&=(\varphi\otimes\psi\otimes \id)((\Delta_\fS\otimes \id)(V))\\
&=\varphi\ast \psi(V).
\end{align*}
The fact that $\pi_V$ is completely bounded follows from the completely isometric isomorphism in Lemma \ref{L:LnBK}.

For $\pi\in \C\B(\fS_*, \B(\K))$, the identification given in Lemma \ref{L:LnBK} determines an operator
$V_\pi\in \fS\ol\otimes \B(\K)$ via $\pi(\varphi)=(\varphi\otimes \id)(V_\pi)$ for all $\varphi\in\fS_*$. Reversing the above proof,
we see that $V_\pi$ is a corepresentation of $\fS$.

In order to get \eqref{E:Vpi}, it suffices to make use of Theorem \ref{T:ww} stating that $\fS$ has property $\bA_1(1)$.
\end{proof}

It is time to look at some examples.

\begin{eg}
\label{Eg:W}
Let $W\in\B(\H_n\otimes\H_n)$ be defined by $W(\xi_u\otimes \xi_v)=\xi_{vu}\otimes \xi_v$. Then
one can show that $W$ is a proper isometry and intertwines $L_w\otimes L_w$ and $\id\otimes L_w$:
\[
(L_w\otimes L_w)W=W(\id\otimes L_w)\quad \mbox{for all}\quad w\in \Fn.
\]
Also one can check that $W$ is a corepresentation of $\L_n$. In fact, checking on the basis vectors, it is easy to see that
it satisfies the identity required in Definition \ref{D:corep}. In order to check $W\in \L_n\ol\otimes \B(\H_n)$, it suffices to
check that $W$ commutes each element in the commutant $(\L_n\ol\otimes \B(\H_n))'=\R_n\otimes \bC I$.
For $R_u\in\R_n$ ($u\in \Fn)$, then
a simple calculation shows that $W(R_u\otimes \id)=(R_u\otimes \id)W$. We are done.

It is not hard to check that the completely bounded representation $\pi_W$ of ${\L_n}_*$ corresponding to the above $W$ in Theorem \ref{T:corep}
is nothing but the left multiplication representation:
\[
\pi_W(\varphi)\xi_u=\varphi(L_u)\xi_u \quad (\varphi\in {\L_n}_*).
\]
\end{eg}

Two points are worth being mentioned here. Firstly, in the group case, the above corepresentation $W$ is known as the fundamental operator
in the theory of Hopf von Neumann algebras.
 Secondly, in spite of the fact that the above $W$ is a proper isometry,  the tuples
$[L_1\otimes L_1\,\cdots\, L_n\otimes L_n]$ and $[\id \otimes L_1\,\cdots\, \id\otimes L_n]$ are unitarily equivalent
because both of them are pure of multiplicity $\infty$ (cf. \cite{Pop89} and Section \ref{S:Ln}).

\begin{eg}
\label{Eg:Lw}
Fix $w\in \Fn$. Let $\rho_w: {\L_n}_*\to \bC$ be defined by $\rho_w(\varphi)=\varphi(L_w)$. Then it is a character
of ${\L_n}_*$. In fact, for all
$\varphi,\psi\in {\L_n}_*$, we have
\begin{align*}
\rho_w(\varphi\ast\psi)
&=\varphi\ast \psi(L_w)=\varphi\otimes\psi(\Delta(L_w))\\
&=\varphi\otimes\psi(L_w\otimes L_w)=\varphi(L_w)\psi(L_w)\\
&=\rho_w(\varphi)\rho_w(\psi).
\end{align*}
In particular, it is a completely contractive representation.
It is easy to check that its corresponding corepresentation $V_{\rho_w}$ of $\L_n$ in Theorem \ref{T:corep}
is given by
\[
V_{\rho_w}=L_w\otimes \id_\bC\cong L_w.
\]
\end{eg}

In what follows, as an amusing application of Theorem \ref{T:corep}, we identify the spectrum of $\fS_*$.
The proof also takes advantage of the useful idea of Fourier expansions of operators in $\L_n$ and $\L_n\ol\otimes\L_n$.

\begin{thm}
\label{T:char}
Let $\fS$ be an analytic free semigroup algebra generated by an isometric $n$-tuple. Then
the (Gelfand) spectrum of $\fS_*$ is $\Fn$.
\end{thm}

\begin{proof}
Obviously, it suffices to show that the spectrum $\Sigma_{{\L_n}_*}$ of ${\L_n}_*$ is $\Fn$ .

By Example \ref{Eg:Lw} above,  one has $\Fn\subseteq \Sigma_{{\L_n}_*}$.
In order to prove $\Sigma_{{\L_n}_*}\subseteq \Fn$, we first
observe that every operator $A\in \L_n\ol\otimes \L_n$ is completely determined by the image $A(\xi_\mt\otimes\xi_\mt)$ of
the ``vacuum vector'' $\xi_\varnothing\otimes \xi_\varnothing$ under $A$.
Indeed, assume that
\[
A(\xi_\varnothing\otimes \xi_\varnothing)=\sum_{u,v\in\Fn}a_{u,v}\, \xi_u\otimes \xi_v
\]
with $a_{u,v}\in\bC$.
Then for all $\alpha,\beta\in\Fn$, one has
\begin{align*}
A(\xi_\alpha\otimes \xi_\beta)
&=AR_{ \alpha}\otimes R_{\beta}(\xi_\varnothing\otimes \xi_\varnothing)\\
&=R_{\alpha}\otimes R_{\beta}A(\xi_\varnothing\otimes \xi_\varnothing)\\
&=R_{ \alpha}\otimes R_{\beta}(\sum_{u,v}a_{u,v}\, \xi_u\otimes \xi_v)\\
&=\sum_{u,v}a_{u,v}\, \xi_{u\widetilde\alpha}\otimes \xi_{v\widetilde\beta},
\end{align*}
where $\widetilde\alpha$ is the word in $\Fn$ obtained by reversing the order of $\alpha$,
and similarly for $\widetilde\beta$.
So similar to $\L_n$, we heuristically view $A$ as its Fourier expansion:
\[
A\sim \sum_{u,v\in\Fn}a_{u,v}\, L_u\otimes L_v,
\]
where the $(u,v)$-th Fourier coefficient is given by $a_{u,v}=(A(\xi_{\mt}\otimes\xi_{\mt}),\xi_{u}\otimes\xi_{v})$.

Now assume that $\varphi: {\L_n}_*\to \bC$ is a character of ${\L_n}_*$. Let $c_{1,1}^\varphi\in\L_n$ be the coefficient operator defined
by
\[
\langle f, c_{1,1}^\varphi\rangle = \varphi(f) \quad (f\in{\L_n}_*).
\]
Since $\varphi$ is a character, it is automatically a completely bounded representation. By Theorem \ref{T:corep}, we have that
$V:=c_{1,1}^\varphi\in\L_n$ is a corepresentation of $\L_n$. That is,
\begin{align}
\label{E:V}
V_{1,3}V_{2,3}= (\Delta\otimes \id_\bC)(V).
\end{align}
Since $V\in\L_n$, one can assume that it has the following Fourier expansion:
\[
V\sim\sum_{w\in\Fn}a_w\,L_w.
\]
Then it follows from \eqref{E:V} that
\begin{align}
\label{E:AW}
\sum_w a_wL_w\otimes \sum_w a_wL_w\ \sim\ \sum_w a_w(L_w\otimes L_w)\in\L_{n}\ol\otimes\L_n.
\end{align}
Then taking the $(w, w)$-th Fourier coefficients of both sides of \eqref{E:AW} gives
\begin{align*}
&((\sum_w a_wL_w\otimes \sum_w a_wL_w)(\xi_\mt\otimes\xi_\mt), \xi_w\otimes\xi_w)\\
&=(\sum_w a_w(L_w\otimes L_w)(\xi_\mt\otimes\xi_\mt), \xi_w\otimes\xi_w).
\end{align*}
This implies $a_w=0$ or $1$ for every $w\in\Fn$. On one hand, there is a word $w\in\Fn$ such that $a_w\ne 0$ as $\varphi\ne 0$ implies $V\ne 0$.
On the other hand, suppose that
there are $u\ne v\in\Fn$ such that both $a_u$ and $a_v$ are non-zero. Then
$a_u=a_v=1$.
Taking the $(u, v)$-th Fourier coefficients of both sides of \eqref{E:AW}, we have
\begin{align*}
1=a_ua_v&=((\sum_w a_wL_w\otimes \sum_w a_wL_w)(\xi_\mt\otimes\xi_\mt), \xi_u\otimes\xi_v)\\
&=(\sum_w a_w(L_w\otimes L_w)(\xi_\mt\otimes\xi_\mt), \xi_u\otimes\xi_v)\\
&=a_u\delta_{u,v}=0.
\end{align*}
This is an obvious contradiction. So $V=L_w$ for some $w\in\Fn$.
Therefore, $\Sigma_{{\L_n}_*}\subseteq \Fn$.
\end{proof}

\subsection*{Tensor products of completely bounded representations of $\fS_*$}
This subsection borrows from \cite[Section 7]{ER1}, and is rather sketched.
Refer to \cite[Section 7]{ER1} for details.

 Suppose that  $\pi_i: \fS_*\to \B(\H_i)$ ($i=1,2$) are two completely bounded representations of $\fS_*$.
Define their multiplication $\pi_1\times \pi_2$ as the composition of the following mappings

\begin{align*}
\fS_*\stackrel{\Delta_{\fS_*}}\longrightarrow \fS_*\stackrel{\rm{eh}}\otimes\fS_*\stackrel{\pi_1\otimes\pi_2}\longrightarrow &\B(\H_1)\stackrel{\rm{eh}}\otimes\B(\H_2)\\
\subseteq  &\B(\H_1)\stackrel{\sigma\rm{h}}\otimes\B(\H_2)\stackrel{\theta}\to \B(\H_1\otimes\H_2),
\end{align*}
where $\theta$ is determined by the product of
\[
\B(\H_1)\to \B(\H_1\otimes\H_2): S\mapsto S\otimes I_{\H_2}
\]
and
\[
\B(\H_2)\to \B(\H_1\otimes\H_2): T\mapsto  I_{\H_1}\otimes T.
\]

Let
\[
\C(\fS_*)=\left\{c_{\xi,\eta}^\pi: c_{\xi,\eta}^\pi(f)=(\pi(f)\xi, \eta),  \ f\in\fS_*,\,  \xi, \, \eta\in\H_\pi\right\}
\]
be the set of coefficient operators of completely bounded representations of $\fS_*$.
Then we have the following duality result.

\begin{thm}
\label{T:dual}
Suppose that $\fS$ is an analytic free semigroup algebra.
Then $\C(\fS_*)=\fS$.
\end{thm}

\begin{proof}
This directly follows from \cite[Theorem 7.2]{ER1} and Example \ref{Eg:Lw}.
\end{proof}

\medskip

Let us end this paper with two remarks.
Firstly, most of the results on Hopf convolution dual algebras of this paper
hold true for more general ones, e.g., those induced from Hopf dual algebras
$\A$ with the property $(\A\ol\otimes\A)_*\cong\A_*\widehat\otimes\A_*$.
Secondly, as we have seen from above, there is a certain analogy between the Fourier algebra $A(\bF_n)$ and ${\L_n}_*$.
So, basically, for whatever property $A(\bF_n)$ has, one could ask if it holds true for ${\L_n}_*$ as long as it makes sense.
This and more will be investigated in the future.

\medskip
\noindent
\textbf{Acknowledgements.} I would like to thank my colleague Prof. Zhiguo Hu for several useful conversations
at the early stage of this work, and Prof. Matthew Kennedy for some useful discussions after
I gave a talk at the CMS 2011 Winter Meeting, in which some results of this paper were
presented. Thanks also go to Prof. Laurent Marcoux for showing me some references.

\end{document}